\documentclass[a4paper,11pt]{article}  
\usepackage{pstricks}
\usepackage{amsmath}
\usepackage{amssymb}
\usepackage{theorem} 
\usepackage{euscript}
\usepackage{exscale,relsize}
\usepackage{graphicx}
\usepackage{srcltx}
\usepackage{xcolor}
\usepackage{url}
\newcommand{\yosi}[2]{\ensuremath{\sideset{^{#2}}{}{\operatorname{}}\!\!#1}}
\newcommand{\minimize}[2]{\ensuremath{\underset{\substack{{#1}}}%
{\mathrm{minimize}}\;\;#2 }}
 
\newcommand{\Scal}[2]{\bigg\langle{#1}\;\bigg|\:{#2}\bigg\rangle} 
\newcommand{\scal}[2]{{\left\langle{{#1}\mid{#2}}\right\rangle}}
\newcommand{\menge}[2]{\big\{{#1}~\big |~{#2}\big\}} 
\newcommand{\Menge}[2]{\left\{{#1}~\Big|~{#2}\right\}}

\newcommand{\KKK}{\ensuremath{\boldsymbol{\mathcal K}}}

\newcommand{\HH}{\ensuremath{{\mathcal H}}}
\newcommand{\GG}{\ensuremath{{\mathcal G}}}
\newcommand{\emp}{\ensuremath{{\varnothing}}}

\newcommand{\Id}{\ensuremath{\operatorname{Id}}\,}
\newcommand{\cart}{\ensuremath{\raisebox{-0.5mm}{\mbox{\LARGE{$\times$}}}}}
\newcommand{\Cart}{\ensuremath{\raisebox{-0.5mm}{\mbox{\huge{$\times$}}}}}
\newcommand{\RR}{\ensuremath{\mathbb{R}}}
\newcommand{\RP}{\ensuremath{\left[0,+\infty\right[}}

\newcommand{\RPP}{\ensuremath{\left]0,+\infty\right[}}

\newcommand{\RPX}{\ensuremath{\left[0,+\infty\right]}}
\newcommand{\RX}{\ensuremath{\left]-\infty,+\infty\right]}}
\newcommand{\RXX}{\ensuremath{\left[-\infty,+\infty\right]}}
\newcommand{\NN}{\ensuremath{\mathbb N}}
\newcommand{\KK}{\ensuremath{\mathcal K}}

\newcommand{\weakly}{\ensuremath{\:\rightharpoonup\:}}
\newcommand{\exi}{\ensuremath{\exists\,}}
\newcommand{\ran}{\ensuremath{\operatorname{ran}}}
\newcommand{\zer}{\ensuremath{\operatorname{zer}}}
\newcommand{\pinf}{\ensuremath{{+\infty}}}

\newcommand{\dom}{\ensuremath{\operatorname{dom}}}
\newcommand{\prox}{\ensuremath{\operatorname{prox}}}

\newcommand{\gra}{\ensuremath{\operatorname{gra}}}

\newcommand{\sri}{\ensuremath{\operatorname{sri}}}
\newcommand{\reli}{\ensuremath{\operatorname{ri}}}
\newcommand{\infconv}{\ensuremath{\mbox{\small$\,\square\,$}}}
\newcommand{\zeroun}{\ensuremath{\left]0,1\right[}}

\newtheorem{theorem}{Theorem}[section]

\newtheorem{proposition}[theorem]{Proposition}

\theoremstyle{plain}{\theorembodyfont{\rmfamily}%
\newtheorem{example}[theorem]{Example}}
\theoremstyle{plain}{\theorembodyfont{\rmfamily}%
\newtheorem{remark}[theorem]{Remark}}
\theoremstyle{plain}{\theorembodyfont{\rmfamily}%
}
\theoremstyle{plain}{\theorembodyfont{\rmfamily}%
}
\theoremstyle{plain}{\theorembodyfont{\rmfamily}%
}
\theoremstyle{plain}{\theorembodyfont{\rmfamily}%
\newtheorem{problem}[theorem]{Problem}}

\numberwithin{equation}{section}
\begin{document}

\title{\sffamily\huge Primal-dual splitting algorithm for solving
inclusions with mixtures of composite, Lipschitzian, and 
parallel-sum type monotone operators\footnote{Contact author: 
P. L. Combettes, {\ttfamily plc@math.jussieu.fr},
phone: +33 1 4427 6319, fax: +33 1 4427 7200.
This work was supported by the Agence Nationale de la 
Recherche under grants ANR-08-BLAN-0294-02 and ANR-09-EMER-004-03.}}

\author{Patrick L. Combettes$^1$ and Jean-Christophe 
Pesquet$^2$\\[5mm]
\small $\!^1$UPMC Universit\'e Paris 06\\
\small Laboratoire Jacques-Louis Lions -- UMR CNRS 7598\\
\small 75005 Paris, France\\
\small \url{plc@math.jussieu.fr}\\[4mm]
\small $\!^2$Universit\'e Paris-Est\\
\small Laboratoire d'Informatique Gaspard Monge -- UMR CNRS 8049\\
\small 77454 Marne la Vall\'ee Cedex 2, France\\
\small \url{jean-christophe.pesquet@univ-paris-est.fr}
}

\date{~}

\maketitle

\vskip 8mm

\begin{abstract} \noindent
We propose a primal-dual splitting algorithm for solving monotone
inclusions involving a mixture of sums, linear compositions, and 
parallel sums of set-valued and Lipschitzian operators. 
An important feature of the algorithm is
that the Lipschitzian operators present in the formulation can
be processed individually via explicit steps, while the set-valued 
operators are processed individually via their resolvents. In 
addition, the algorithm is highly parallel in that most of its 
steps can be executed simultaneously. This work brings together and 
notably extends various types of structured monotone inclusion 
problems and their solution methods. The application to 
convex minimization problems is given special attention.
\end{abstract} 

{\bfseries Keywords} 
maximal monotone operator, 
monotone inclusion, 
nonsmooth convex optimization,
parallel sum, 
set-valued duality,
splitting algorithm 

{\bfseries Mathematics Subject Classifications (2010)} 
47H05, 49M29, 49M27, 90C25, 49N15.

\maketitle

\newpage

\section{Introduction}

Duality theory occupies a central place in classical optimization
\cite{Fenc53,Gols71,More66,Rock67,Rock74}. Since the mid 1960s it 
has expanded in various directions, e.g., variational 
inequalities \cite{Aldu05,Ekel99,Gaba83,Gohc02,Konn03,Mosc72}, 
minimax and saddle point problems 
\cite{Lebe67,Mcli74,More64,Rock64}, and,
from a more global perspective, monotone inclusions
\cite{Atto96,Reic05,Siop11,Ecks99,Merc80,Penn00,Robi99}.
In the present paper, we propose an algorithm for solving the 
following structured duality framework for monotone inclusions
that encompasses the above cited works.
In this formulation, we denote by $B\infconv D$ the parallel sum 
of two set-valued operators $B$ and $D$ (see \eqref{e:parasum}).
This operation plays a central role in convex analysis and monotone 
operator theory. In particular, $B\infconv D$ can be seen as a
regularization of $B$ by $D$, and $\infconv$ is naturally connected 
to addition through duality since $(B+D)^{-1}=B^{-1}\infconv D^{-1}$.
It is also strongly related to the infimal convolution of functions
through subdifferentials.
We refer the reader to \cite{Livre1,Luqu86,Moud96,Pass86,Seeg90} 
and the references therein for background on the parallel sum.

\begin{problem}
\label{prob:1}
Let $\HH$ be a real Hilbert space, let $z\in\HH$, let $m$ be a 
strictly positive integer, 
let $A\colon\HH\to 2^{\HH}$ be maximally monotone, and let
$C\colon\HH\to\HH$ be monotone and $\mu$-Lipschitzian for some
$\mu\in\RPP$. For every $i\in\{1,\ldots,m\}$, let $\GG_i$ be a 
real Hilbert space, let $r_i\in\GG_i$, let 
$B_i\colon\GG_i\to 2^{\GG_i}$ 
be maximally monotone, let $D_i\colon\GG_i\to 2^{\GG_i}$ 
be monotone and such that $D_i^{-1}$ is $\nu_i$-Lipschitzian, for 
some $\nu_i\in\RPP$, and suppose that $L_i\colon\HH\to\GG_i$ is 
a nonzero bounded linear operator. 
The problem is to solve the primal inclusion
\begin{equation}
\label{e:mprimal}
\text{find}\quad\overline{x}\in\HH\;\;\text{such that}\;\;
z\in A\overline{x}+\sum_{i=1}^mL_i^*\big((B_i\infconv D_i)
(L_i\overline{x}-r_i)\big)+C\overline{x},
\end{equation}
together with the dual inclusion
\begin{equation}
\label{e:mdual}
\text{find}\;\overline{v_1}\in\GG_1,\:\ldots,\:
\overline{v_m}\in\GG_m
\;\:\text{such that}\;
(\exi x\in\HH)\;
\begin{cases}
z-\sum_{i=1}^mL_i^*\overline{v_i}\in Ax+Cx\\
(\forall i\in\{1,\ldots,m\})\;\overline{v_i}
\in(B_i\infconv D_i)(L_ix-r_i).
\end{cases}
\end{equation}
\end{problem}

Problem~\ref{prob:1} captures and extends various existing problem
formulations. Here are some examples that illustrate its
versatility and the breadth of its scope. 

\begin{example}
\label{ex:2011-06-22i}
In Problem~\ref{prob:1} set
\begin{equation}
\label{e:2011-06-22}
m=1,\quad
C\colon x\mapsto 0,\quad\text{and}\quad
D_1\colon y\mapsto
\begin{cases}
\GG_1,&\text{if}\;\;y=0;\\
\emp,&\text{if}\;\;y\neq 0.
\end{cases}
\end{equation}
Then we recover a duality framework investigated in 
\cite{Siop11,Ecks99,Penn00,Robi99}, namely (we drop
the subscript `1' for brevity),
\begin{equation}
\label{e:2011-06-16b}
\text{find}\;\;(\overline{x},\overline{v})\in\HH\oplus\GG\;\;
\text{such that}\;
\begin{cases}
z\in A\overline{x}+L^*\big(B(L\overline{x}-r)\big)\\
-r\in -L\big(A^{-1}(z-L^*\overline{v})\big)+B^{-1}\overline{v}.
\end{cases}
\end{equation}
\end{example}

\begin{example}
\label{ex:2011-06-22ii}
In Example~\ref{ex:2011-06-22i}, let $\GG=\HH$, $r=z=0$, and 
$L=\Id$. Then we obtain the duality setting of 
\cite{Atto96,Merc80}, i.e.,
\begin{equation}
\label{e:2011-06-22c}
\text{find}\;\;(\overline{x},\overline{u})\in\HH\oplus\HH\;\;
\text{such that}\;
\begin{cases}
0\in A\overline{x}+B\overline{x}\\
0\in-A^{-1}(-\overline{u})+B^{-1}\overline{u}.
\end{cases}
\end{equation}
The special case of variational inequalities was first treated in
\cite{Mosc72}.
\end{example}

\begin{example}
\label{ex:2011-06-22iii}
In Example~\ref{ex:2011-06-22i}, let $A$ and $B$ be the 
subdifferentials of lower semicontinuous convex functions 
$f\colon\HH\to\RX$ and $g\colon\GG\to\RX$, respectively. Then, 
under suitable constraint qualification, we obtain the classical 
Fenchel-Rockafellar duality framework \cite{Rock67}, i.e.,
\begin{equation}
\label{e:2010-11-18f}
\begin{cases}
\minimize{x\in\HH}{f(x)+g(Lx-r)-\scal{x}{z}}\\[3mm]
\minimize{v\in\GG}{f^*(z-L^*v)+g^*(v)+\scal{v}{r}}.
\end{cases}
\end{equation}
\end{example}

\begin{example}
\label{ex:2011-06-22iv}
In Problem~\ref{prob:1}, set $C\colon x\mapsto 0$, 
$z=0$, and $(\forall i\in\{1,\ldots,m\})$ $\GG_i=\HH$, 
$r_i=0$, $L_i=\Id$, and $D_i=\rho_i^{-1}\Id$, where 
$\rho_i\in\RPP$. Then 
it follows from \cite[Proposition~23.6(ii)]{Livre1} that,
for every $i\in\{1,\ldots,m\}$, 
$B_i\infconv D_i=(\Id-J_{\rho_iB_i})/\rho_i=\yosi{B_i}{\rho_i}$
is the Yosida approximation of index $\rho_i$ of $B_i$. Thus,
\eqref{e:mprimal} reduces to
\begin{equation}
\label{e:mprimal-1}
\text{find}\quad\overline{x}\in\HH\;\;\text{such that}\quad
0\in A\overline{x}+\sum_{i=1}^m\yosi{B_i}{\rho_i}\overline{x}.
\end{equation}
This primal problem is investigated in \cite[Section~6.3]{Opti04}.
In the case when $m=1$, we obtain the primal-dual problem (we drop
the subscript `1' for brevity)
\begin{equation}
\label{e:2011-06-16z}
\text{find}\;\;(\overline{x},\overline{u})\in\HH\oplus\HH\;\;
\text{such that}\;
\begin{cases}
0\in A\overline{x}+\yosi{B}{\rho}\overline{x}\\
0\in -A^{-1}(-\overline{u})+B^{-1}\overline{u}+\rho\overline{u}
\end{cases}
\end{equation}
investigated in \cite{Reic05}.
\end{example}

\begin{example}
\label{ex:2011-06-30}
In Problem~\ref{prob:1}, set $m=1$, $\GG_1=\GG$, $L_1=L$,
$z=0$, and $r_1=0$, and
let $A$ and $B_1$ be the subdifferentials of lower semicontinuous 
convex functions $f\colon\HH\to\RX$ and $g\colon\GG\to\RX$, 
respectively. In addition, let $C$ be the gradient of a 
differentiable convex function $h\colon\HH\to\RR$, 
and let $D$ be the subdifferential of a lower semicontinuous
strongly convex function $\ell\colon\GG\to\RX$.
Then, under suitable constraint qualification, \eqref{e:mprimal}
assumes the form of the minimization problem
\begin{equation}
\minimize{x\in\HH}{f(x)+(g\infconv\ell)(Lx)+h(x)},
\end{equation}
which can be rewritten as
\begin{equation}
\minimize{x\in\HH,\,y\in\GG}{f(x)+h(x)+g(y)+\ell(Lx-y)}.
\end{equation}
In the special case when $h=0$, $\GG=\HH$, $L=\Id$, and $\ell$ is a
quadratic coupling function, such formulations have been 
investigated in \cite{Acke80,Atto07,Aujo05,Cham97,Smms05}.
\end{example}

\begin{example}
\label{ex:2011-06-22v}
In Problem~\ref{prob:1}, set $m=1$, $\GG_1=\HH$, $L_1=\Id$, 
$B_1=D_1^{-1}\colon x\mapsto\{0\}$, and $z=r_1=0$. Then 
\eqref{e:mprimal} yields the inclusion 
$0\in A\overline{x}+C\overline{x}$ studied in \cite{Tsen00},
where an algorithm using explicit steps for $C$ was proposed.
\end{example}

\begin{example}
\label{ex:2011-06-22vi}
In Problem~\ref{prob:1}, set $A\colon x\mapsto\{0\}$ and $C=\Id$. 
Furthermore, for every $i\in\{1,\ldots,m\}$, let $B_i$ be the 
subdifferential of a lower semicontinuous convex function 
$g_i\colon\GG_i\to\RX$ and let $D_i^{-1}\colon y\mapsto\{0\}$.
Then, under suitable constraint qualification, we obtain the 
primal-dual pair considered in \cite{Bang11}, namely
\begin{equation}
\label{e:pvn}
\minimize{x\in\HH}{\sum_{i=1}^mg_i(L_ix-r_i)+
\frac{1}{2}\|x-z\|^2}
\end{equation}
and
\begin{equation}
\label{e:dvn}
\minimize{v_1\in\GG_1,\ldots,\, v_m\in\GG_m}
{\frac12\left\|z-\sum_{i=1}^mL_i^*v_i\right\|^2+\sum_{i=1}^m
\big(g_i^*(v_i)+\scal{v_i}{r_i}\big)}.
\end{equation}
\end{example}

\begin{example}
\label{ex:2011-06-22vii}
The special case of Problem~\ref{prob:1} in which 
\begin{equation}
A\colon x\mapsto\{0\},\; C\colon x\mapsto 0,
\quad\text{and}\quad
(\forall i\in\{1,\ldots,m\})\quad
D_i\colon y\mapsto
\begin{cases}
\GG_i,&\text{if}\;\;y=0;\\
\emp,&\text{if}\;\;y\neq 0
\end{cases}
\end{equation}
yields the primal-dual pair 
\begin{equation}
\label{e:2010-11-11x}
\text{find}\quad\overline{x}\in\HH\;\;\text{such that}\quad
z\in\sum_{i=1}^mL_i^*\big(B_i(L_i\overline{x}-r_i)\big)
\end{equation}
and
\begin{equation}
\label{e:eqdualprod}
\text{find}\;\:\overline{v_1}\in\GG_1,\ldots,\overline{v_m}
\in\GG_m\quad\text{such that}\;
\begin{cases}
\sum_{i=1}^mL_i^*\overline{v_i}=z\\
(\exi x\in\HH)(\forall i\in\{1,\ldots,m\})\;\:
\overline{v_i}\in B_i(L_ix-r_i).
\end{cases}
\end{equation}
This framework is considered in \cite[Theorem~3.8]{Siop11}.
\end{example}

Conceptually, the primal problem \eqref{e:mprimal} could be
recast in the form of \eqref{e:2010-11-11x}, namely 
\begin{equation}
\label{e:mprimal-2}
\text{find}\quad\overline{x}\in\HH\;\;\text{such that}\quad
z\in\sum_{i=0}^mL_i^*\big(E_i(L_i\overline{x}-r_i)\big),
\end{equation}
where
\begin{multline}
\GG_0=\HH,\;
E_0=A+C,\;
L_0=\Id,\;
r_0=0,
\quad\text{and}\quad
(\forall i\in\{1,\ldots,m\})\quad E_i=B_i\infconv D_i.
\end{multline}
In turn, one could contemplate the possibility of using the 
primal-dual algorithm proposed in \cite[Theorem~3.8]{Siop11} to 
solve Problem~\ref{prob:1}. However, this algorithm
requires the computation of the resolvents of the operators 
$A+C$ and $(B_i^{-1}+D_i^{-1})_{1\leq i\leq m}$, which are usually
intractable. Thus, for numerical purposes, Problem~\ref{prob:1} 
cannot be reduced to Example~\ref{ex:2011-06-22vii}.
Let us stress that, even in the instance of the simple 
inclusion $0\in A\overline{x}+C\overline{x}$, it is precisely 
the objective of the forward-backward splitting algorithm and 
its variants \cite{Livre1,Smms05,Merc79,Tsen91,Tsen00} to circumvent
the computation of the resolvent of $A+C$, as would impose a naive 
application of the proximal point algorithm \cite{Roc76a}. 

The goal of this paper is to propose a fully split algorithm for 
solving Problem~\ref{prob:1} that employs the operators 
$A$, $(L_i)_{1\leq i\leq m}$, $(B_i)_{1\leq i\leq m}$, 
$(D_i)_{1\leq i\leq m}$, and $C$ separately. 
An important feature of the algorithm is to activate
the single-valued operators $(L_i)_{1\leq i\leq m}$,
$(D_i^{-1})_{1\leq i\leq m}$, and $C$ through explicit steps.
In addition, it exhibits a highly parallel structure which allows 
for the simultaneous activation of the operators involved. 
This new splitting method goes significantly beyond the 
state-of-the-art, which is limited to specific subclasses of 
Problem~\ref{prob:1}.

In Section~\ref{sec:2}, we briefly set our notation. The new
splitting method is proposed in Section~\ref{sec:3}, where we
also prove its convergence. The special case of minimization 
problems is discussed in Section~\ref{sec:4}.

\section{Notation and background}
\label{sec:2}
Our notation is standard. We refer the reader to 
\cite{Livre1,Zali02} for background on convex analysis and monotone 
operator theory. Hereafter, $\KK$ is a real Hilbert space.

We denote the scalar product of a Hilbert space by 
$\scal{\cdot}{\cdot}$ and the associated norm by $\|\cdot\|$. 
The symbols $\weakly$ and $\to$ denote 
respectively weak and strong convergence. Moreover,
$\GG_1\oplus\cdots\oplus\GG_m$ is the Hilbert direct sum of the 
Hilbert spaces $(\GG_i)_{1\leq i\leq m}$ in Problem~\ref{prob:1}, 
i.e., their product space equipped with the norm
$(y_i)_{1\leq i\leq m}\mapsto\sqrt{\sum_{i=1}^m\|y_i\|^2}$. 
For every $i\in\{1,\ldots,m\}$, let $T_i$ be a mapping from 
$\GG_i$ to some set ${\mathcal R}$. Then
\begin{equation}
\label{e:2011-06-30a}
\bigoplus_{i=1}^mT_i\colon\bigoplus_{i=1}^m\GG_i\to{\mathcal R}
\colon(y_i)_{1\leq i\leq m}\mapsto\sum_{i=1}^mT_iy_i.
\end{equation}

Let $M\colon\KK\to 2^{\KK}$ be a set-valued operator.
We denote by $\ran M=\menge{u\in\KK}{(\exi x\in\KK)\;u\in Mx}$ 
the range of $M$, by $\dom M=\menge{x\in\KK}{Mx\neq\emp}$ its
domain, by $\zer M=\menge{x\in\KK}{0\in Mx}$ its set of zeros, 
by $\gra M=\menge{(x,u)\in\KK\times\KK}{u\in Mx}$ its 
graph, and by $M^{-1}$ its inverse, i.e., the set-valued operator 
with graph
$\menge{(u,x)\in\KK\times\KK}{u\in Mx}$. The resolvent of $M$ is
\begin{equation}
J_M=(\Id+M)^{-1}, 
\end{equation}
where $\Id$ denotes the identity operator on $\KK$.
Moreover, $M$ is monotone if
\begin{equation}
(\forall (x,u)\in\gra M)(\forall (y,v)\in\gra M)\quad
\scal{x-y}{u-v}\geq 0,
\end{equation}
and maximally so if there exists no monotone operator 
$\widetilde{M}\colon\KK\to 2^{\KK}$ such that $\gra M\subset\gra
\widetilde{M}\neq\gra M$. 
\label{l:2009-09-20i}
We say that $M$ is uniformly monotone at $x\in\dom M$ if there 
exists an increasing function $\phi\colon\RP\to\RPX$ that vanishes 
only at $0$ such that
\begin{equation}
\label{e:Unifmon}
(\forall u\in Mx)(\forall (y,v)\in\gra M)\quad
\scal{x-y}{u-v}\geq\phi(\|x-y\|).
\end{equation}
The parallel sum of two set-valued operators $M_1$ and $M_2$ 
from $\KK$ to $2^{\KK}$ is 
\begin{equation}
\label{e:parasum}
M_1\infconv M_2=(M_1^{-1}+M_2^{-1})^{-1}.
\end{equation}

We denote by $\Gamma_0(\KK)$ the class of lower semicontinuous 
convex functions $\varphi\colon\KK\to\RX$ such that 
$\dom \varphi=\menge{x\in\KK}{\varphi(x)<\pinf}\neq\emp$. Now let
$\varphi\in\Gamma_0(\KK)$. 
The conjugate of $\varphi$ is the function 
$\varphi^*\in\Gamma_0(\KK)$ defined by 
$\varphi^*\colon u\mapsto\sup_{x\in\KK}(\scal{x}{u}-\varphi(x))$,
and the subdifferential of $\varphi$ is the maximally monotone 
operator
\begin{equation}
\partial\varphi\colon\KK\to 2^{\KK}\colon x\mapsto
\menge{u\in\KK}{(\forall y\in\KK)\;\:\scal{y-x}{u}+\varphi(x)\leq
\varphi(y)}
\end{equation}
with inverse given by
\begin{equation}
\label{e:heidelberg2011-07-03}
(\partial\varphi)^{-1}=\partial\varphi^*.
\end{equation}
Moreover, for every $x\in\KK$, $\varphi+\|x-\cdot\|^2/2$ possesses 
a unique minimizer, which is denoted by $\prox_\varphi x$. We have
\begin{equation}
\label{e:prox2}
\prox_\varphi=J_{\partial \varphi}.
\end{equation}
We say that $\varphi$ is $\nu$-strongly convex for some $\nu\in\RPP$
if $\varphi-\nu\|\cdot\|^2/2$ is convex, and that $\varphi$ is 
uniformly convex at $x\in\dom\varphi$ if there exists an increasing 
function $\phi\colon\RP\to\RPX$ that vanishes only at $0$ such that
\begin{equation}
\label{e:Unifconvex}
(\forall y\in\dom\varphi)(\forall\alpha\in\zeroun)\quad
\varphi(\alpha x+(1-\alpha)y)+\alpha(1-\alpha)\phi(\|x-y\|)\leq
\alpha\varphi(x)+(1-\alpha)\varphi(y).
\end{equation}
The infimal convolution of two functions $\varphi_1$ and 
$\varphi_2$ from $\KK$ to $\RX$ is
\begin{equation}
\label{e:infconv1}
\varphi_1\infconv\varphi_2\colon\KK\to\RXX\colon x\mapsto
\inf_{y\in\KK}\big(\varphi_1(y)+\varphi_2(x-y)\big).
\end{equation}

Finally, let $S$ be a convex subset of $\KK$. The strong 
relative interior of $S$, i.e., the set of points $x\in S$ such 
that the cone generated by $-x+S$ is a closed vector subspace of 
$\KK$, is denoted by $\sri S$, and the relative interior of $S$, 
i.e., the set of points $x\in S$ such that the cone 
generated by $-x+S$ is a vector subspace of $\KK$, is 
denoted by $\reli S$.

\section{Main result}
\label{sec:3}

Our main result is the following theorem, which presents our
new splitting algorithm and describes its asymptotic behavior. 

\begin{theorem}
\label{t:1}
In Problem~\ref{prob:1}, suppose that 
\begin{equation}
\label{e:2011-06-24a}
z\in\ran\bigg(A+\sum_{i=1}^mL_i^*(B_i\infconv D_i)
(L_i\cdot-r_i)+C\bigg).
\end{equation}
Let $(a_{1,n})_{n\in\NN}$, $(b_{1,n})_{n\in\NN}$, and
$(c_{1,n})_{n\in\NN}$ be absolutely summable sequences in $\HH$
and, for every $i \in \{1,\ldots,m\}$, let 
$(a_{2,i,n})_{n\in\NN}$, $(b_{2,i,n})_{n\in\NN}$, and
$(c_{2,i,n})_{n\in\NN}$ be absolutely summable sequences in 
$\GG_i$. Furthermore, set 
\begin{equation}
\label{e:2011-06:25b}
\beta=\max\{\mu,\nu_1,\ldots,\nu_m\}+\sqrt{\sum_{i=1}^m\|L_i\|^2},
\end{equation}
let $x_0\in\HH$, let
$(v_{1,0},\ldots,v_{m,0})\in\GG_1\oplus\cdots\oplus\GG_m$, let 
$\varepsilon\in\left]0,1/(\beta+1)\right[$, 
let $(\gamma_n)_{n\in\NN}$ be a sequence in 
$[\varepsilon,(1-\varepsilon)/\beta]$, and set
\begin{equation}
\label{e:blackpage}
(\forall n\in\NN)\quad 
\begin{array}{l}
\left\lfloor
\begin{array}{l}
y_{1,n}=x_n-\gamma_n\big(Cx_n+\sum_{i=1}^mL_i^*v_{i,n}
+a_{1,n}\big)\\
p_{1,n}=J_{\gamma_n A}(y_{1,n}+\gamma_nz)+b_{1,n}\\
\operatorname{For}\;i=1,\ldots,m\\
\left\lfloor
\begin{array}{l}
y_{2,i,n}=v_{i,n}+\gamma_n\big(L_ix_n-D_i^{-1}v_{i,n}
+a_{2,i,n}\big)\\
p_{2,i,n}=J_{\gamma_n B_i^{-1}}(y_{2,i,n}-\gamma_nr_i)
+b_{2,i,n}\\
q_{2,i,n}=p_{2,i,n}+\gamma_n\big(L_ip_{1,n}-D_i^{-1}p_{2,i,n}
+c_{2,i,n}\big)\\
v_{i,n+1}=v_{i,n}-y_{2,i,n}+q_{2,i,n}.
\end{array}
\right.\\[1mm]
q_{1,n}=p_{1,n}-\gamma_n\big(Cp_{1,n}+\sum_{i=1}^mL_i^*p_{2,i,n}
+c_{1,n}\big)\\
x_{n+1}=x_n-y_{1,n}+q_{1,n}.
\end{array}
\right.\\
\end{array}
\end{equation}
Then the following hold.
\begin{enumerate}
\item
\label{t:1i}
$\sum_{n\in\NN}\|x_n-p_{1,n}\|^2<\pinf$ and 
$(\forall i\in\{1,\ldots,m\})$
$\sum_{n\in\NN}\|v_{i,n}-p_{2,i,n}\|^2<\pinf$.
\item
\label{t:1ii}
There exist a solution $\overline{x}$ to \eqref{e:mprimal} and a 
solution $(\overline{v_1},\ldots,\overline{v_m})$ to 
\eqref{e:mdual} such that the following hold.
\begin{enumerate}
\item
\label{t:1iia}
$z-\sum_{j=1}^mL_j^*\overline{v_j}\in A\overline{x}+C\overline{x}$ 
and $(\forall i\in\{1,\ldots,m\})$
$L_i\overline{x}-r_i\in B_i^{-1}\overline{v_i}
+D_i^{-1}\overline{v_i}$.
\item
\label{t:1iib}
$(\forall i\in\{1,\ldots,m\})$
$-r_i\in -L_i\big((A^{-1}\infconv C^{-1})\big(z-\sum_{j=1}^m
L_j^*\overline{v_j}\big)\big)+B_i^{-1}\overline{v_i}
+D_i^{-1}\overline{v_i}$.
\item
\label{t:1iic}
$x_n\weakly\overline{x}$ and $p_{1,n}\weakly\overline{x}$.
\item
\label{t:1iid}
$(\forall i\in\{1,\ldots,m\})$ $v_{i,n}\weakly\overline{v_i}$ and 
$p_{2,i,n}\weakly\overline{v_i}$.
\item 
\label{t:1iie}
Suppose that $A$ or $C$ is uniformly monotone at $\overline{x}$.
Then $x_n\to\overline{x}$ and $p_{1,n}\to\overline{x}$.
\item 
\label{t:1iif}
Suppose that, for some $i\in \{1,\ldots,m\}$, $B_i^{-1}$ or 
$D_i^{-1}$ is uniformly monotone at $\overline{v_i}$. Then
$v_{i,n}\to\overline{v_i}$ and $p_{2,i,n}\to\overline{v_i}$.
\end{enumerate}
\end{enumerate}
\end{theorem}
\begin{proof}
Let us first rewrite \eqref{e:blackpage} as
\begin{equation}
\label{e:blackpagepart1}
(\forall n\in\NN)\quad 
\begin{array}{l}
\left\lfloor
\begin{array}{l}
y_{1,n}=x_n-\gamma_n\big(Cx_n+\sum_{i=1}^mL_i^*v_{i,n}
+a_{1,n}\big)\\
\operatorname{For}\;i=1,\ldots,m\\
\left\lfloor
\begin{array}{l}
y_{2,i,n}=v_{i,n}+\gamma_n\big(L_ix_n-D_i^{-1}v_{i,n}
+a_{2,i,n}\big)\\
\end{array}
\right.\\[1mm]
p_{1,n}=J_{\gamma_n A}(y_{1,n}+\gamma_nz)+b_{1,n}\\
\operatorname{For}\;i=1,\ldots,m\\
\left\lfloor
\begin{array}{l}
p_{2,i,n}=J_{\gamma_n B_i^{-1}}(y_{2,i,n}-\gamma_nr_i)+b_{2,i,n}\\
\end{array}
\right.\\[1mm]
q_{1,n}=p_{1,n}-\gamma_n\big(Cp_{1,n}+\sum_{i=1}^mL_i^*p_{2,i,n}
+c_{1,n}\big)\\
\operatorname{For}\;i=1,\ldots,m\\
\left\lfloor
\begin{array}{l}
q_{2,i,n}=p_{2,i,n}+\gamma_n\big(L_ip_{1,n}-D_i^{-1}p_{2,i,n}
+c_{2,i,n}\big)\\
\end{array}
\right.\\[1mm]
x_{n+1}=x_n-y_{1,n}+q_{1,n}\\
\operatorname{For}\;i=1,\ldots,m\\
\left\lfloor
\begin{array}{l}
v_{i,n+1}=v_{i,n}-y_{2,i,n}+q_{2,i,n}.
\end{array}
\right.\\[1mm]
\end{array}
\right.\\
\end{array}
\end{equation}
Next, let us introduce the Hilbert space
\begin{equation}
\label{e:2011-06-24}
\KKK=\HH\oplus\GG_1\oplus\cdots\oplus\GG_m,
\end{equation}
and the operators
\begin{equation}
\label{e:2011-06-18b}
\begin{array}{rll}
\boldsymbol{M}\colon\KKK&\to& 2^{\KKK}\\
(x,v_1,\ldots,v_m)&\mapsto&
(-z+Ax)\times(r_1+B_1^{-1}v_1)\times\cdots\times(r_m+B_m^{-1}v_m)
\end{array}
\end{equation}
and
\begin{equation}
\label{e:2011-06-18c}
\begin{array}{rll}
\boldsymbol{Q}\colon\KKK&\to&\KKK\\
(x,v_1,\ldots,v_m)&\mapsto&
\big(Cx+L_1^*v_1+\cdots+L_m^*v_m,-L_1x+D_1^{-1}v_1,\ldots,-L_mx
+D_m^{-1}v_m\big).
\end{array}
\end{equation}
Since the operator $A$ and $(B_i)_{1\leq i\leq m}$ are maximally 
monotone, so is $\boldsymbol{M}$ by 
\cite[Propositions~20.22 and~20.23]{Livre1}.
In addition, \cite[Propositions~23.15(ii) and~23.16]{Livre1} yield
\begin{multline}
\label{e:2011-06-25a}
(\forall\gamma\in\RPP)(\forall (x,v_1,\ldots,v_m)\in\KKK)\\
J_{\gamma\boldsymbol{M}}(x,v_1,\ldots,v_m)=
\big(J_{\gamma A}(x+\gamma z),
J_{\gamma B_1^{-1}}(v_1-\gamma r_1),\ldots,
J_{\gamma B_m^{-1}}(v_m-\gamma r_m)\big).
\end{multline}
Let us now examine the properties of ${\boldsymbol Q}$.
To this end, let $(x,v_1,\ldots,v_m)$ and $(y,w_1,\ldots,w_m)$ be 
two points in $\KKK$. Using the monotonicity of the operators $C$ 
and $(D_i^{-1})_{1\leq i\leq m}$, we derive from 
\eqref{e:2011-06-18c} that
\begin{multline}
\label{e:2011-06-19a}
\scal{(x,v_1,\ldots,v_m)-(y,w_1,\ldots,w_m)}
{{\boldsymbol Q}(x,v_1,\ldots,v_m)-
{\boldsymbol Q}(y,w_1,\ldots,w_m)}\\
\begin{aligned}[b]
&=\big\langle(x-y,v_1-w_1,\ldots,v_m-w_m)\,\big|\,
\big(Cx-Cy+L_1^*(v_1-w_1)+\cdots+L_m^*(v_m-w_m),\\
&\qquad\qquad\qquad\;-L_1(x-y)+D_1^{-1}v_1-D_1^{-1}w_1,
\ldots,-L_m(x-y)+D_m^{-1}v_m-D_m^{-1}w_m\big)\big\rangle\\
&=\scal{x-y}{Cx-Cy}+\sum_{i=1}^m
\scal{v_i-w_i}{D_i^{-1}v_i-D_i^{-1}w_i}\\
&\quad\;+\sum_{i=1}^m\big(\scal{x-y}{L_i^*(v_i-w_i)}-
\scal{v_i-w_i}{L_i(x-y)}\big)\\
&=\scal{x-y}{Cx-Cy}+
\sum_{i=1}^m\scal{v_i-w_i}{D_i^{-1}v_i-D_i^{-1}w_i}\\
&\geq 0.
\end{aligned}
\end{multline}
Hence, ${\boldsymbol Q}$ is monotone. Using the triangle inequality,
the Lipschitzianity assumptions, the Cauchy-Schwarz inequality, and 
\eqref{e:2011-06:25b}, we obtain
\begin{multline}
\label{e:2011-06-19b}
\big\|{\boldsymbol Q}(x,v_1,\ldots,v_m)-
{\boldsymbol Q}(y,w_1,\ldots,w_m)\big\|\\
\begin{aligned}[b]
&=\bigg\|\Big(Cx-Cy,D_1^{-1}v_1-D_1^{-1}w_1,\ldots,
D_m^{-1}v_m-D_m^{-1}w_m\Big)\\
&\qquad\;+\bigg(\sum_{i=1}^mL_i^*(v_i-w_i),-L_1(x-y),\ldots,
-L_m(x-y)\bigg)\bigg\|\\
&\leq\bigg\|\Big(Cx-Cy,D_1^{-1}v_1-D_1^{-1}w_1,\ldots,
D_m^{-1}v_m-D_m^{-1}w_m\Big)\bigg\|\\
&\quad\;+\bigg\|\bigg(\sum_{i=1}^mL_i^*(v_i-w_i),-L_1(x-y),\ldots,
-L_m(x-y)\bigg)\bigg\|\\
&=\sqrt{\|Cx-Cy\|^2+\sum_{i=1}^m\big
\|D_i^{-1}v_i-D_i^{-1}w_i\big\|^2}+
\sqrt{\bigg\|\sum_{i=1}^mL_i^*(v_i-w_i)\bigg\|^2+
\sum_{i=1}^m\|L_i(x-y)\|^2}\\
&\leq\sqrt{\mu^2\|x-y\|^2+\sum_{i=1}^m\nu_i^2\|v_i-w_i\|^2}
+\sqrt{\bigg(\sum_{i=1}^m\|L_i\|\,\|v_i-w_i\|\bigg)^2+
\sum_{i=1}^m\|L_i\|^2\,\|x-y\|^2}\\
&\leq\max\{\mu,\nu_1,\ldots,\nu_m\}
\sqrt{\|x-y\|^2+\sum_{i=1}^m\|v_i-w_i\|^2}\\
&\quad\;+\sqrt{\bigg(\sum_{i=1}^m\|L_i\|^2\bigg)
\bigg(\sum_{i=1}^m\|v_i-w_i\|^2\bigg)+
\bigg(\sum_{i=1}^m\|L_i\|^2\bigg)\|x-y\|^2}\\
&=\beta\|(x,v_1,\ldots,v_m)-(y,w_1,\ldots,w_m)\|.
\end{aligned}
\end{multline}
To sum up, we have shown that
\begin{equation}
\label{e:2011-04-19f}
\boldsymbol{M}\;\text{is maximally monotone and}\;
\boldsymbol{Q}\;\text{is monotone and $\beta$-Lipschitzian.}
\end{equation}
Next, let us observe that
\begin{align}
\eqref{e:2011-06-24a}
&\Leftrightarrow (\exi x\in\HH)\quad
z\in Ax+\sum_{i=1}^mL_i^*\big((B_i\infconv D_i)(L_ix-r_i)\big)
+Cx\nonumber\\
&\Leftrightarrow (\exi(x,v_1,\ldots,v_m)\in\KKK)\quad
\begin{cases}
z\in Ax+\sum_{i=1}^mL_i^*v_i+Cx\\
v_1\in(B_1\infconv D_1)(L_1x-r_1)\\
~~~~~\vdots\\
v_m\in(B_m\infconv D_m)(L_mx-r_m)
\end{cases}
\nonumber\\
&\Leftrightarrow (\exi(x,v_1,\ldots,v_m)\in\KKK)\quad
\begin{cases}
0\in -z+Ax+\sum_{i=1}^mL_i^*v_i+Cx\\
0\in r_1+B_1^{-1}v_1+D_1^{-1}v_1-L_1x\\
~~~\vdots\\
0\in r_m+B_m^{-1}v_m+D_m^{-1}v_m-L_mx\\
\end{cases}
\nonumber\\
&\Leftrightarrow (\exi(x,v_1,\ldots,v_m)\in\KKK)\quad
(0,\ldots,0)\in
(-z+Ax)\times(r_1+B_1^{-1}v_1)\times\cdots\times(r_m+B_m^{-1}v_m)
\nonumber\\
&\hskip 42mm +(L_1^*v_1+\cdots+L_m^*v_m+Cx,D_1^{-1}v_1-L_1x,\ldots,
D_m^{-1}v_m-L_mx)
\nonumber\\
&\Leftrightarrow (\exi(x,v_1,\ldots,v_m)\in\KKK)\quad
(0,\ldots,0)\in(\boldsymbol{M}+\boldsymbol{Q})(x,v_1,\ldots,v_m).
\end{align}
In other words, 
\begin{equation}
\label{e:2011-06-19g}
\zer(\boldsymbol{M}+\boldsymbol{Q})\neq\emp.
\end{equation}
Now, let us set
\begin{equation}
\label{e:nyc2010-10-31y'}
(\forall n\in\NN)\quad
\begin{cases}
\boldsymbol{x}_n=(x_n,v_{1,n},\ldots,v_{1,n})\\
\boldsymbol{y}_n=(y_{1,n},y_{2,1,n},\ldots,y_{2,m,n})\\
\boldsymbol{p}_n=(p_{1,n},p_{2,1,n},\ldots,p_{2,m,n})\\
\boldsymbol{q}_n=(q_{1,n},q_{2,1,n},\ldots,q_{2,m,n})
\end{cases}
\quad\text{and}\qquad
\begin{cases}
\boldsymbol{a}_n=(a_{1,n},a_{2,1,n},\ldots,a_{2,m,n})\\
\boldsymbol{b}_n=(b_{1,n},b_{2,1,n},\ldots,b_{2,m,n})\\
\boldsymbol{c}_n=(c_{1,n},c_{2,1,n},\ldots,c_{2,m,n}).
\end{cases}
\end{equation}
We first observe that our assumptions imply that
\begin{equation}
\label{e:broome-st2010-10-27g}
\sum_{n\in\NN}\|\boldsymbol{a}_n\|<\pinf,\quad
\sum_{n\in\NN}\|\boldsymbol{b}_n\|<\pinf,
\quad\text{and}\quad
\sum_{n\in\NN}\|\boldsymbol{c}_n\|<\pinf.
\end{equation}
Furthermore, it follows from \eqref{e:2011-06-18c}, 
\eqref{e:2011-06-25a}, and \eqref{e:nyc2010-10-31y'}, 
that \eqref{e:blackpagepart1} assumes in $\KKK$ the form of the
error-tolerant forward-backward-forward algorithm
\begin{equation}
\label{e:blackpagepart2}
(\forall n\in\NN)\quad 
\begin{array}{l}
\left\lfloor
\begin{array}{l}
\boldsymbol{y}_n=\boldsymbol{x}_n-
\gamma_n(\boldsymbol{Q}\boldsymbol{x}_n+\boldsymbol{a}_n)\\
\boldsymbol{p}_n=J_{\gamma_n\boldsymbol{M}}\,\boldsymbol{y}_n
+\boldsymbol{b}_n\\
\boldsymbol{q}_n=\boldsymbol{p}_n-\gamma_n(\boldsymbol{Q}
\boldsymbol{p}_n+\boldsymbol{c}_n)\\
\boldsymbol{x}_{n+1}=\boldsymbol{x}_n-
\boldsymbol{y}_n+\boldsymbol{q}_n.
\end{array}
\right.\\[2mm]
\end{array}
\end{equation}

\ref{t:1i}: It follows from \eqref{e:2011-04-19f}, 
\eqref{e:2011-06-19g}, \eqref{e:broome-st2010-10-27g}, 
\eqref{e:blackpagepart2}, and \cite[Theorem~2.5(i)]{Siop11} that 
$\sum_{n\in\NN}\|\boldsymbol{x}_n-\boldsymbol{p}_n\|^2<\pinf$.

\ref{t:1ii}: 
It follows from \cite[Theorem~2.5(ii)]{Siop11} that there exists 
$\overline{\boldsymbol x}\in\zer({\boldsymbol M}+{\boldsymbol Q})$ 
such that 
\begin{equation}
\label{e:2011-06-26w}
{\boldsymbol x}_n\weakly\overline{\boldsymbol x}
\quad\text{and}\quad
{\boldsymbol p}_n\weakly\overline{\boldsymbol x}. 
\end{equation}
Let us set
\begin{equation}
\label{e:2011-06-26v}
\overline{\boldsymbol x}=
(\overline{x},\overline{v_1},\ldots,\overline{v_m}). 
\end{equation}
In view of \eqref{e:2011-06-18b} and \eqref{e:2011-06-18c}, 
\begin{align}
\overline{\boldsymbol x}\in\zer(\boldsymbol{M}+\boldsymbol{Q})
&\Leftrightarrow 
\begin{cases}
0\in -z+A\overline{x}+\sum_{i=1}^mL_i^*\overline{v_i}+
C\overline{x}\\
0\in r_1+B_1^{-1}\overline{v_1}+D_1^{-1}\overline{v_1}-
L_1\overline{x}\\
~~~\vdots\\
0\in r_m+B_m^{-1}\overline{v_m}+D_m^{-1}\overline{v_m}-
L_m\overline{x}
\end{cases}
\nonumber\\
&\Leftrightarrow 
\begin{cases}
z-\sum_{j=1}^mL_j^*\overline{v_j}\in A\overline{x}+C\overline{x}\\
L_1\overline{x}-r_1\in(B_1^{-1}+D_1^{-1})\overline{v_1}\\
\hskip 17mm \vdots\\
L_m\overline{x}-r_m\in(B_m^{-1}+D_m^{-1})\overline{v_m}
\end{cases}
\label{e:2011-06-27a}\\
&\Leftrightarrow 
\begin{cases}
z-\sum_{j=1}^mL_j^*\overline{v_j}\in A\overline{x}+C\overline{x}\\
\overline{v_1}\in(B_1\infconv D_1)(L_1\overline{x}-r_1)\\
\hskip 6mm \vdots\\
\overline{v_m}\in(B_m\infconv D_m)(L_m\overline{x}-r_m)
\end{cases}
\label{e:2011-06-26}\\
&\Rightarrow 
\begin{cases}
z-\sum_{j=1}^mL_j^*\overline{v_j}\in A\overline{x}+C\overline{x}\\
L_1^*\overline{v_1}\in L_1^*\big((B_1\infconv D_1)
(L_1\overline{x}-r_1)\big)\\
\hskip 10mm \vdots\\
L_m^*\overline{v_m}\in L_m^*\big((B_m\infconv D_m)
(L_m\overline{x}-r_m)\big)
\end{cases}
\nonumber\\
&\Rightarrow 
z\in A\overline{x}+\sum_{i=1}^mL_i^*\big((B_i\infconv D_i)
(L_i\overline{x}-r_i)\big)+C\overline{x}
\nonumber\\
&\Leftrightarrow 
\overline{x}\;\text{solves \eqref{e:mprimal}}.
\label{e:2011-06-26x}
\end{align}
On the other hand, \eqref{e:2011-06-26} means that
\begin{equation}
\label{e:2011-06-26y}
(\overline{v_1},\ldots,\overline{v_m})\;\
\text{solves \eqref{e:mdual}}.
\end{equation}

\ref{t:1iia}: This follows from \eqref{e:2011-06-27a}.

\ref{t:1iib}: We derive from \eqref{e:2011-06-27a} that
\begin{equation}
\label{e:2011-06-27i}
\overline{x}\in(A+C)^{-1}\bigg(z-\sum_{j=1}^mL_j^*\overline{v_j}
\bigg)
\quad\text{and}\quad(\forall i\in\{1,\ldots,m\})\quad
L_i\overline{x}-r_i\in(B_i^{-1}+D_i^{-1})\overline{v_i}.
\end{equation}
Hence, 
\begin{equation}
\label{e:2011-06-27r}
(\forall i\in\{1,\ldots,m\})\quad
\begin{cases}
-L_i\overline{x}\in -L_i\big((A^{-1}\infconv C^{-1})
\big(z-\sum_{j=1}^mL_j^*\overline{v_j}\big)\big)\\
L_i\overline{x}-r_i\in(B_i^{-1}+D_i^{-1})\overline{v_i}.
\end{cases}
\end{equation}
Thus, 
\begin{equation}
(\forall i\in\{1,\ldots,m\})\quad
-r_i\in-L_i\bigg((A^{-1}\infconv C^{-1})
\bigg(z-\sum_{j=1}^mL_j^*\overline{v_j}\bigg)\bigg)+
B_i^{-1}\overline{v_i}+D_i^{-1}\overline{v_i}.
\end{equation}

\ref{t:1iic}: This follows from \eqref{e:2011-06-26w},
\eqref{e:2011-06-26v}, and \eqref{e:2011-06-26x}.

\ref{t:1iid}: This follows from \eqref{e:2011-06-26w},
\eqref{e:2011-06-26v}, and \eqref{e:2011-06-26y}.

\ref{t:1iie}: Let us set
\begin{equation}
\label{e:trtard}
(\forall n\in\NN)\quad 
\begin{cases}
\widetilde{y}_{1,n}=x_n-\gamma_n\big(Cx_n+
\sum_{j=1}^mL_j^*v_{j,n}\big)\\
\widetilde{p}_{1,n}=J_{\gamma_n A}(\widetilde{y}_{1,n}+\gamma_nz)
\end{cases}
\end{equation}
and
\begin{equation}
\label{e:ford}
(\forall i\in\{1,\ldots,m\})(\forall n\in\NN)\quad 
\begin{cases}
\widetilde{y}_{2,i,n}=v_{i,n}+\gamma_n(L_ix_n-D_i^{-1}v_{i,n})\\
\widetilde{p}_{2,i,n}=J_{\gamma_n B_i^{-1}}
(\widetilde{y}_{2,i,n}-\gamma_nr_i).
\end{cases}
\end{equation}
Then, in view of \eqref{e:blackpage}, 
\begin{equation}
(\forall n\in\NN)\quad
\|y_{1,n}-\widetilde{y}_{1,n}\|\leq\gamma_n\|a_{1,n}\|\leq
\beta^{-1}\|a_{1,n}\|
\end{equation}
and, using the nonexpansiveness of the resolvents 
\cite[Proposition~23.7]{Livre1}, we obtain
\begin{align}
(\forall n\in\NN)\quad
\|p_{1,n}-\widetilde{p}_{1,n}\| 
&\leq\|J_{\gamma_n A}(y_{1,n}+\gamma_nz)+b_{1,n}-
J_{\gamma_n A}(\widetilde{y}_{1,n}+\gamma_nz)\|\nonumber\\
&\leq\|y_{1,n}-\widetilde{y}_{1,n}\|+\|b_{1,n}\|\nonumber\\
&\leq\beta^{-1}\|a_{1,n}\|+\|b_{1,n}\|.
\end{align}
Since the sequences $(a_{1,n})_{n\in \NN}$ and 
$(b_{1,n})_{n\in \NN}$ are absolutely summable, it follows that
\begin{equation}
\label{e:convyp1t}
y_{1,n}-\widetilde{y}_{1,n}\to 0 
\quad\text{and}\quad
p_{1,n}-\widetilde{p}_{1,n}\to 0.
\end{equation}
Using the same arguments, we derive from \eqref{e:blackpage}
and \eqref{e:ford} that
\begin{equation}
\label{e:convyp2t}
(\forall i\in\{1,\ldots,m\})\quad
y_{2,i,n}-\widetilde{y}_{2,i,n}\to 0 \quad\text{and}\quad
p_{2,i,n}-\widetilde{p}_{2,i,n}\to 0.
\end{equation}
On the other hand, we deduce from \ref{t:1iia} that
there exists $u\in\HH$ such that
\begin{equation}
\label{e:uAxdusoir}
u\in A\overline{x}\quad\text{and}\quad
z=u+\sum_{j=1}^mL_j^*\overline{v_j}+C\overline{x},
\end{equation}
and that
\begin{equation}
\label{e:uBxdusoir}
(\forall i\in\{1,\ldots,m\})\quad L_i\overline{x}-r_i-
D_i^{-1}\overline{v_i}\in B_i^{-1}\overline{v_i}.
\end{equation}
In addition, \eqref{e:trtard} yields
\begin{equation}
\label{e:inclundusoir}
(\forall n\in\NN)\quad 
\gamma_n^{-1}(x_n-\widetilde{p}_{1,n})-C x_n
-\sum_{j=1}^m L_j^* v_{j,n}+z\in A \widetilde{p}_{1,n}
\end{equation}
while \eqref{e:ford} yields
\begin{equation}
\label{e:inclundusoir2}
(\forall i\in\{1,\ldots,m\})(\forall n\in\NN)\quad 
\gamma_n^{-1}(v_{i,n}-\widetilde{p}_{2,i,n})+L_i x_n-D_i^{-1}
v_{i,n}-r_i\in B_i^{-1}\widetilde{p}_{2,i,n}.
\end{equation}
Now let us set
\begin{equation}
\label{e:defalpha12}
(\forall n\in\NN)\quad 
\begin{cases}
\alpha_{1,n}=\|x_n-\widetilde{p}_{1,n}\|
\big(\varepsilon^{-1}\|\widetilde{p}_{1,n}-\overline{x}\|
+\mu\|x_n-\overline{x}\|+\sum_{i=1}^m \|L_i\|\,
\|v_{i,n}-\overline{v_{i}}\|\big)\\
\alpha_{2,n}=\sum_{i=1}^m(\varepsilon^{-1}+\nu_i) 
\|v_{i,n}-\widetilde{p}_{2,i,n}\|\,
\|\widetilde{p}_{2,i,n}-\overline{v_i}\|.
\end{cases}
\end{equation}
It follows from \ref{t:1i}, \ref{t:1iic}, \ref{t:1iid},
\eqref{e:convyp1t}, and \eqref{e:convyp2t} that
\begin{equation}
\label{e:2011-07-09a}
\alpha_{1,n}\to 0\quad\text{and}\quad\alpha_{2,n}\to 0.
\end{equation}
Using the Cauchy-Schwarz inequality, and the Lipschitzianity 
and monotonicity of $C$, we obtain
\begin{multline}
\label{e:majACstrong}
(\forall n\in\NN)\quad
\alpha_{1,n}+\Scal{x_n-\overline{x}}{\sum_{i=1}^m
L_i^*(\overline{v_i}-v_{i,n})}\\
\begin{aligned}[b]
&\geq\|x_n-\widetilde{p}_{1,n}\|
\big(\varepsilon^{-1}\|\widetilde{p}_{1,n}-\overline{x}\|+
\|C x_n-C\overline{x}\|\big)+\Scal{\widetilde{p}_{1,n}-x_n}
{\sum_{i=1}^mL_i^*(\overline{v_i}-v_{i,n})}\\
&\quad\;+\Scal{x_n-\overline{x}}{\sum_{i=1}^mL_i^*
(\overline{v_i}-v_{i,n})}\\
&=\|x_n-\widetilde{p}_{1,n}\|
\big(\varepsilon^{-1} \|\widetilde{p}_{1,n}-\overline{x}\|
+\|C x_n-C \overline{x}\|\big)+
\Scal{\widetilde{p}_{1,n}-\overline{x}}{\sum_{i=1}^m
L_i^*(\overline{v_i}-v_{i,n})}\\
&\geq\Scal{\widetilde{p}_{1,n}-\overline{x}}
{\gamma_n^{-1}(x_n-\widetilde{p}_{1,n})+\sum_{i=1}^m
L_i^*(\overline{v_i}-v_{i,n})}+\scal{\widetilde{p}_{1,n}-x_n}{C
\overline{x}-C x_n}\\
&=\Scal{\widetilde{p}_{1,n}-\overline{x}}
{\gamma_n^{-1}(x_n-\widetilde{p}_{1,n})-\sum_{i=1}^m L_i^* v_{i,n}
-Cx_n+\sum_{i=1}^mL_i^*\overline{v_i}+ C \overline{x}}+ 
\scal{\overline{x}-x_n}{C\overline{x}-C x_n}\\
&=\Scal{\widetilde{p}_{1,n}-\overline{x}}
{\gamma_n^{-1}(x_n-\widetilde{p}_{1,n})-\sum_{i=1}^m L_i^*v_{i,n}
-C x_n+z-u}+\scal{\overline{x}-x_n}{C\overline{x}-C x_n}\\
&\geq\Scal{\widetilde{p}_{1,n}-\overline{x}}
{\bigg(\gamma_n^{-1}(x_n-\widetilde{p}_{1,n})-
\sum_{i=1}^mL_i^*v_{i,n}-C x_n+z\bigg)-u}.
\end{aligned}
\end{multline}
Now suppose that $A$ is uniformly monotone at $\overline{x}$.
Then, in view of \eqref{e:uAxdusoir}, \eqref{e:inclundusoir}, 
and \eqref{e:majACstrong}, there exists an increasing function 
$\phi_A\colon\RP\to\RPX$ that vanishes only at $0$ such that
\begin{equation}
\label{e:ineq1strong}
(\forall n\in\NN)\quad
\alpha_{1,n}+\Scal{x_n-\overline{x}}{\sum_{i=1}^m
L_i^*(\overline{v_i}-v_{i,n})}\geq\phi_A
(\|\widetilde{p}_{1,n}-\overline{x}\|).
\end{equation}
On the other hand, it follows from \eqref{e:defalpha12}, the 
Lipschitzianity of the operators 
$(D_i^{-1})_{1\leq i\leq m}$, \eqref{e:uBxdusoir}, 
\eqref{e:inclundusoir2}, and the monotonicity of the operators 
$(B_i^{-1})_{1\leq i\leq m}$ and $(D_i^{-1})_{1\leq i\leq m}$ that
\begin{multline}
\label{e:ineq2strong}
(\forall n\in\NN)\quad
\alpha_{2,n}+\Scal{x_n-\overline{x}}{\sum_{i=1}^m
L_i^*(\widetilde{p}_{2,i,n}-\overline{v_{i}})}\\
\begin{aligned}[b]
&\geq\sum_{i=1}^m\scal{\gamma_n^{-1}(v_{i,n}-\widetilde{p}_{2,i,n})
-D_i^{-1} v_{i,n}+D_i^{-1} \widetilde{p}_{2,i,n}
+ L_i(x_n-\overline{x})}{\widetilde{p}_{2,i,n}-\overline{v_i}}\\
&=\sum_{i=1}^m\Big(\scal{\gamma_n^{-1}(v_{i,n}-
\widetilde{p}_{2,i,n})+L_i x_n-D_i^{-1} v_{i,n}-r_i-(L_i\overline{x}
-r_i-D_i^{-1}\overline{v_i})}{\widetilde{p}_{2,i,n}-\overline{v_i}}\\
&\qquad\quad+\scal{D_i^{-1}\widetilde{p}_{2,i,n}-D_i^{-1}
\overline{v_i}}{\widetilde{p}_{2,i,n}-\overline{v_i}}\Big)\\
&\geq 0.
\end{aligned}
\end{multline}
Adding \eqref{e:ineq1strong} and \eqref{e:ineq2strong} yields
\begin{equation}
\label{e:conclustrong}
(\forall n\in\NN)\quad
\alpha_{1,n}+\alpha_{2,n}+\Scal{x_n-\overline{x}}{\sum_{i=1}^m
L_i^*(\widetilde{p}_{2,i,n}-v_{i,n})}
\geq\phi_A(\|\widetilde{p}_{1,n}-\overline{x}\|).
\end{equation}
It then follows from \eqref{e:2011-07-09a}, \ref{t:1iic},
\ref{t:1i}, \eqref{e:convyp2t}, and \cite[Lemma~2.41(iii)]{Livre1} 
that $\phi_A(\|\widetilde{p}_{1,n}-\overline{x}\|)\to 0$ and,
in turn, that $\widetilde{p}_{1,n}\to\overline{x}$.
Hence, in view of \ref{t:1i} and \eqref{e:convyp1t}, we get
$x_n\to\overline{x}$ and $p_{1,n}\to\overline{x}$.
Likewise, if $C$ is uniformly monotone at $\overline{x}$,
there exists an increasing function $\phi_C\colon\RP\to\RPX$ that 
vanishes only at $0$ such that 
\begin{align}
(\forall n\in\NN)\quad
\alpha_{1,n}+\alpha_{2,n}+\Scal{x_n-\overline{x}}{\sum_{i=1}^m
L_i^*(\widetilde{p}_{2,i,n}-v_{i,n})}\geq
\phi_C(\|x_n-\overline{x}\|),
\end{align}
and we reach the same conclusion.

\ref{t:1iif}: Suppose that $B_i^{-1}$ is uniformly monotone
at $\overline{v_i}$ for some $i\in \{1,\ldots,m\}$. Then, 
proceeding as in \eqref{e:ineq2strong}, there exists an increasing 
function $\phi_{B_i}\colon\RP\to\RPX$ that vanishes only at $0$ 
such that
\begin{multline}
\label{e:ineq2strongdual}
(\forall n\in\NN)\quad
\alpha_{2,n}+\Scal{x_n-\overline{x}}{\sum_{j=1}^m
L_j^*(\widetilde{p}_{2,j,n}-\overline{v_{j}})}\\
\begin{aligned}[b]
&\geq\sum_{j=1}^m\bigg(\scal{\gamma_n^{-1}(v_{j,n}
-\widetilde{p}_{2,j,n})+L_j x_n-D_j^{-1} v_{j,n}-r_j-
(L_j\overline{x}-r_j-D_j^{-1}\overline{v_j})}
{\widetilde{p}_{2,j,n}-\overline{v_j}}\\
&\qquad\qquad+\scal{D_j^{-1}\widetilde{p}_{2,j,n}-
D_j^{-1}\overline{v_j}}
{\widetilde{p}_{2,j,n}-\overline{v_j}}\bigg)\\
&\geq\sum_{j=1}^m\scal{\gamma_n^{-1}(v_{j,n}
-\widetilde{p}_{2,j,n})+L_j x_n-D_j^{-1} v_{j,n}-r_j-
(L_j\overline{x}-r_j-D_j^{-1}\overline{v_j})}
{\widetilde{p}_{2,j,n}-\overline{v_j}}\\
&\geq\scal{\gamma_n^{-1}(v_{i,n}
-\widetilde{p}_{2,i,n})+L_i x_n-D_i^{-1} v_{i,n}-r_i-
(L_i\overline{x}-r_i-D_i^{-1}\overline{v_i})}
{\widetilde{p}_{2,i,n}-\overline{v_i}}\\
&\geq\phi_{B_i}(\|\widetilde{p}_{2,i,n}-\overline{v_i}\|).
\end{aligned}
\end{multline}
On the other hand, according to \eqref{e:majACstrong},
\begin{equation}
\label{e:majACstrongdual}
(\forall n\in\NN)\quad
\alpha_{1,n}+\Scal{x_n-\overline{x}}{\sum_{j=1}^m
L_j^*(\overline{v_j}-v_{j,n})}\geq 0.
\end{equation}
Hence,
\begin{equation}
\label{e:conclustrongdual }
(\forall n\in\NN)\quad
\alpha_{1,n}+\alpha_{2,n}+\Scal{x_n-\overline{x}}{\sum_{j=1}^m
L_j^*(\widetilde{p}_{2,j,n}-v_{j,n})}
\geq\phi_{B_i}(\|\widetilde{p}_{2,i,n}-\overline{v_i}\|).
\end{equation}
By proceeding as previously, we infer that
$\widetilde{p}_{2,i,n}\to\overline{v_i}$ and hence, via
\eqref{e:convyp2t} and \ref{t:1i}, that $p_{2,i,n}\to\overline{v_i}$ 
and $v_{i,n}\to\overline{v_i}$. If $D_i^{-1}$ is uniformly
monotone at $\overline{v_i}$, the same arguments lead
to these conclusions.
\end{proof}

In the following remarks, we comment on the structure of
the proposed algorithm and its relation to existing work.

\begin{remark}
\label{r:2011-06-29}
Here are some observations regarding the structure of algorithm
\eqref{e:blackpage}. 
\begin{enumerate}
\item
The algorithm achieves full splitting in that each of the operators 
appearing in Problem~\ref{prob:1} is used separately.
\item
The algorithm uses explicit steps for the single-valued operators 
and implicit steps for the set-valued operators. Since explicit steps
are typically much easier to implement than implicit steps, the 
algorithm therefore exploits efficiently the properties of the 
operators.
\item
The sequences $(a_{1,n})_{n\in\NN}$, $(b_{1,n})_{n\in\NN}$, and
$(c_{1,n})_{n\in\NN}$, and, for every $i\in \{1,\ldots,m\}$,
$(a_{2,i,n})_{n\in\NN}$, $(b_{2,i,n})_{n\in\NN}$, and
$(c_{2,i,n})_{n\in\NN}$ relax the requirement for exact evaluations 
of the operators over the course of the iterations. 
\item
Most of the elementary steps in \eqref{e:blackpage} can be executed 
in parallel.
\item
The update of the variable $p_{2,i,n}$ can also be carried out
using the resolvent of $B_i$ since \cite[Proposition~23.18]{Livre1}
$J_{\gamma_nB_i^{-1}}(y_{2,i,n}-\gamma_nr_i)=y_{2,i,n}-
\gamma_nr_i-\gamma_n J_{\gamma_n^{-1}B_i}
(\gamma_n^{-1}(y_{2,i,n}-\gamma_nr_i))$.
\end{enumerate}
\end{remark}

\begin{remark}
\label{r:2011-06-19}
Some noteworthy connections between Theorem~\ref{t:1} and existing 
work are the following.
\begin{enumerate}
\item
Unlike most splitting methods, the proposed algorithm is designed
to solve explicitly a dual problem.
\item
In the special case when $m=1$ and $D_1$ is as in 
\eqref{e:2011-06-22}, the primal problem \eqref{e:mprimal}
reduces to (we drop the subscript `1' for brevity)
\begin{equation}
\label{e:2011-06-29p}
\text{find}\quad\overline{x}\in\HH\;\;\text{such that}\quad
z\in A\overline{x}+L^*\big(B(L\overline{x}-r)\big)+C\overline{x},
\end{equation}
the dual problem \eqref{e:mdual} reduces to
\begin{equation}
\label{e:2011-06-29q}
\text{find}\quad\overline{v}\in\GG\;\;\text{such that}\quad
-r\in-L\big((A+C)^{-1}(z-L^*\overline{v})\big)+
B^{-1}\overline{v},
\end{equation}
and the algorithm is governed by the iteration
\begin{equation}
\label{e:blackpagepart5}
\begin{array}{l}
\left\lfloor
\begin{array}{l}
y_{1,n}=x_n-\gamma_n\big(Cx_n+L^*v_{n}
+a_{1,n}\big)\\
y_{2,n}=v_{n}+\gamma_n(Lx_n+a_{2,n})\\
p_{1,n}=J_{\gamma_n A}(y_{1,n}+\gamma_nz)+b_{1,n}\\
p_{2,n}=J_{\gamma_n B^{-1}}(y_{2,n}-\gamma_nr)+b_{2,n}\\
q_{1,n}=p_{1,n}-\gamma_n\big(Cp_{1,n}+L^*p_{2,n}+c_{1,n}\big)\\
q_{2,n}=p_{2,n}+\gamma_n(Lp_{1,n}+c_{2,n})\\
x_{n+1}=x_n-y_{1,n}+q_{1,n}\\
v_{n+1}=v_n-y_{2,n}+q_{2,n}.
\end{array}
\right.\\
\end{array}
\end{equation}
On the one hand, if $C\colon x\mapsto 0$, we recover the primal-dual
setting of \cite{Siop11} and its algorithm 
(\cite[Eq.~(3.1)]{Siop11}). On the other hand, if 
$L\colon x\mapsto 0$, $B\colon y\mapsto\{0\}$, $z=0$, and $r=0$,
\eqref{e:2011-06-29p} yields the problem studied in 
\cite{Tsen00}, and \eqref{e:blackpagepart5} without error terms
and dual variables yields a primal algorithm proposed in that paper,
namely 
\begin{equation}
\label{e:tseng2000}
\begin{array}{l}
\left\lfloor
\begin{array}{l}
y_{1,n}=x_n-\gamma_n Cx_n\\
p_{1,n}=J_{\gamma_n A}y_{1,n}\\
q_{1,n}=p_{1,n}-\gamma_n Cp_{1,n}\\
x_{n+1}=x_n-y_{1,n}+q_{1,n}.
\end{array}
\right.\\
\end{array}
\end{equation}
Let us note that, even when we specialize \eqref{e:2011-06-29p} 
to $\GG=\HH$ and $L=\Id$, there does not appear to exist an 
alternative algorithm that splits $A$, $B$, and $C$ and uses 
explicit steps on the Lipschitzian operator $C$.
\item
When $C\colon x\mapsto 0$ and, for every $i\in\{1,\ldots,m\}$, 
$D_i^{-1}\colon y\mapsto\{0\}$, we recover the primal-dual setting
of \cite[Theorem~3.8]{Siop11}. However, the algorithm we obtain is
different from that proposed in that paper, and novel. 
\item
In general, the weak convergence results of
Theorem~\ref{t:1}\ref{t:1ii} cannot be improved to strong 
convergence without additional hypotheses on the operators
such as those described in \ref{t:1iie} and \ref{t:1iif}. 
Indeed, in the special case when \eqref{e:mprimal} reduces to 
the problem of finding a zero of $A$, the primal component of 
\eqref{e:blackpage} reduces to 
the proximal point algorithm, namely 
(set $C\colon x\mapsto 0$ in \eqref{e:tseng2000})
\begin{equation}
\label{e:2011-06-30b}
(\forall n\in\NN)\quad x_{n+1}=J_{\gamma_n A}x_n,
\end{equation}
which is known to converge weakly but not strongly 
\cite{Baus05,Gule91}.
\end{enumerate}
\end{remark}

\section{Minimization problems}
\label{sec:4}

The proposed monotone operator splitting algorithm can be applied 
to a broader class of problems than that within the reach of 
existing splitting methods. It has therefore potential 
applications in the areas in which these methods have been used, 
e.g., partial differential equations \cite{Gaba83,Merc79}, 
mechanics \cite{Glow89,Merc80}, variational inequalities 
\cite{Livre1,Pang03,Tsen91}, game theory \cite{Luis11}, traffic 
theory \cite{Fuku96}, and evolution equations \cite{Sico10}. In 
this section, we focus on the application of the results of 
Section~\ref{sec:3} to convex minimization problems. 

\begin{problem}
\label{prob:2}
Let $\HH$ be a real Hilbert space, let $z\in\HH$, let $m$ be a 
strictly positive integer, let $f\in\Gamma_0(\HH)$, and let
$h\colon\HH\to\RR$ be convex and differentiable with a 
$\mu$-Lipschitzian gradient for some $\mu\in\RPP$. 
For every $i\in\{1,\ldots,m\}$, let $\GG_i$ be a 
real Hilbert space, let $r_i\in\GG_i$, let 
$g_i\in\Gamma_0(\GG_i)$, let $\ell_i\in\Gamma_0(\GG_i)$ 
be $1/\nu_i$-strongly convex, for 
some $\nu_i\in\RPP$, and suppose that $L_i\colon\HH\to\GG_i$ is 
a nonzero bounded linear operator. Consider the problem
\begin{equation}
\label{e:primalvar}
\minimize{x\in\HH}{f(x)+\sum_{i=1}^m\,(g_i\infconv\ell_i)
(L_ix-r_i)+h(x)-\scal{x}{z}}, 
\end{equation}
and the dual problem
\begin{equation}
\label{e:dualvar}
\minimize{v_1\in\GG_1,\ldots,v_m\in\GG_m}{
\big(f^*\infconv h^*\big)\bigg(z-\sum_{i=1}^mL_i^*v_i\bigg)
+\sum_{i=1}^m\big(g_i^*(v_i)+\ell_i^*(v_i)+\scal{v_i}{r_i}\big)}.
\end{equation}
\end{problem}

The following result is an offspring of Theorem~\ref{t:1}.
\begin{theorem}
\label{t:2}
In Problem~\ref{prob:2}, suppose that 
\begin{equation}
\label{e:2011-06-27c}
z\in\ran\bigg(\partial f+\sum_{i=1}^mL_i^*(\partial g_i\infconv
\partial \ell_i)(L_i\cdot-r_i)+\nabla h\bigg).
\end{equation}
Let $(a_{1,n})_{n\in\NN}$, $(b_{1,n})_{n\in\NN}$, and
$(c_{1,n})_{n\in\NN}$ be absolutely summable sequences in $\HH$
and, for every $i\in\{1,\ldots,m\}$, let 
$(a_{2,i,n})_{n\in\NN}$, $(b_{2,i,n})_{n\in\NN}$, and
$(c_{2,i,n})_{n\in\NN}$ be absolutely summable sequences in 
$\GG_i$. Furthermore, set 
\begin{equation}
\label{e:2011-06:27d}
\beta=\max\{\mu,\nu_1,\ldots,\nu_m\}+\sqrt{\sum_{i=1}^m\|L_i\|^2},
\end{equation}
let $x_0\in\HH$, let
$(v_{1,0},\ldots,v_{m,0})\in\GG_1\oplus\cdots\oplus\GG_m$, let 
$\varepsilon\in\left]0,1/(\beta+1)\right[$, 
let $(\gamma_n)_{n\in\NN}$ be a sequence in 
$[\varepsilon,(1-\varepsilon)/\beta]$, and set
\begin{equation}
\label{e:blackpagepart3}
(\forall n\in\NN)\quad 
\begin{array}{l}
\left\lfloor
\begin{array}{l}
y_{1,n}=x_n-\gamma_n\big(\nabla h(x_n)+\sum_{i=1}^mL_i^*v_{i,n}
+a_{1,n}\big)\\
p_{1,n}=\prox_{\gamma_n f}(y_{1,n}+\gamma_nz)+b_{1,n}\\
\operatorname{For}\;i=1,\ldots,m\\
\left\lfloor
\begin{array}{l}
y_{2,i,n}=v_{i,n}+\gamma_n(L_ix_n-\nabla\ell_i^{*}(v_{i,n})
+a_{2,i,n})\\
p_{2,i,n}=\prox_{\gamma_n g_i^{*}}(y_{2,i,n}-\gamma_nr_i)
+b_{2,i,n}\\
q_{2,i,n}=p_{2,i,n}+\gamma_n\big(L_ip_{1,n}-\nabla\ell_i^{*}
(p_{2,i,n})+c_{2,i,n}\big)\\
v_{i,n+1}=v_{i,n}-y_{2,i,n}+q_{2,i,n}.
\end{array}
\right.\\[1mm]
q_{1,n}=p_{1,n}-\gamma_n\big(\nabla h(p_{1,n})+
\sum_{i=1}^mL_i^*p_{2,i,n} +c_{1,n}\big)\\
x_{n+1}=x_n-y_{1,n}+q_{1,n}.
\end{array}
\right.\\
\end{array}
\end{equation}
Then the following hold.
\begin{enumerate}
\item
\label{t:2i}
$\sum_{n\in\NN}\|x_n-p_{1,n}\|^2<\pinf$ and 
$(\forall i\in\{1,\ldots,m\})$
$\sum_{n\in\NN}\|v_{i,n}-p_{2,i,n}\|^2<\pinf$.
\item
\label{t:2ii}
There exist a solution $\overline{x}$ to \eqref{e:primalvar} and a 
solution $(\overline{v_1},\ldots,\overline{v_m})$ to 
\eqref{e:dualvar} such that the following hold.
\begin{enumerate}
\item
$z-\sum_{j=1}^mL_j^*\overline{v_j}\in\partial f(\overline{x})+
\nabla h(\overline{x})$ 
and $(\forall i\in\{1,\ldots,m\})$
$L_i\overline{x}-r_i\in\partial g_i^{*}(\overline{v_i})
+\nabla\ell_i^{*}(\overline{v_i})$.
\item
$x_n\weakly\overline{x}$ and $p_{1,n}\weakly\overline{x}$.
\item
$(\forall i\in\{1,\ldots,m\})$ $v_{i,n}\weakly\overline{v_i}$ and 
$p_{2,i,n}\weakly\overline{v_i}$.
\item 
\label{t:2iie}
Suppose that $f$ or $h$ is uniformly convex at $\overline{x}$.
Then $x_n\to\overline{x}$ and $p_{1,n}\to\overline{x}$.
\item 
\label{t:2iif}
Suppose that, for some $i\in\{1,\ldots,m\}$, $g_i^{*}$ or 
$\ell_i^{*}$ is uniformly convex at $\overline{v_i}$. Then
$v_{i,n}\to \overline{v_i}$ and $p_{2,i,n} \to \overline{v_i}$.
\end{enumerate}
\end{enumerate}
\end{theorem}
\begin{proof}
Let us first establish a connection between Problem~\ref{prob:2} 
and Problem~\ref{prob:1}. To this end, let us define
\begin{equation}
\label{e:2011-06-27m}
A=\partial f,\;\:C=\nabla h,\;\:\text{and}\quad
(\forall i\in\{1,\ldots,m\})\quad
B_i=\partial g_i\;\:\text{and}\;\:
D_i=\partial\ell_i.
\end{equation}
It is clear that \eqref{e:2011-06-27c} yields 
\eqref{e:2011-06-24a} and, using \eqref{e:heidelberg2011-07-03}
and \eqref{e:prox2}, that
\eqref{e:blackpagepart3} yields \eqref{e:blackpage}. Moreover,
it follows from \cite[Theorem~20.40]{Livre1} that the operators 
$A$ and $(B_i)_{1\leq i\leq m}$ are maximally monotone, and from
\cite[Proposition~17.10]{Livre1} that $C$ is monotone. On the 
other hand, for every $i\in\{1,\ldots,m\}$, it follows from 
the $1/\nu_i$-strong convexity of $\ell_i$ and 
\cite[Corollary~13.33 and Theorem~18.15]{Livre1} that 
$\ell_i^*$ is Fr\'echet differentiable on $\GG_i$ with a 
$\nu_i$-Lipschitzian gradient, and from 
\eqref{e:heidelberg2011-07-03} that $D_i^{-1}=\nabla\ell_i^*$.
Altogether, we can apply Theorem~\ref{t:1} to obtain the 
existence of a point $\overline{x}\in\HH$ such that
\begin{equation}
\label{e:2011-06-28b}
z\in\partial f(\overline{x})+
\sum_{i=1}^mL_i^*\big((\partial g_i\infconv\partial\ell_i)
(L_i\overline{x}-r_i)\big)+\nabla h(\overline{x}),
\end{equation}
and of an $m$-tuple $(\overline{v_1},\ldots,\overline{v_m})\in
\GG_1\oplus\cdots\oplus\GG_m$ such that
\begin{equation}
\label{e:mdual-var1}
(\exi x\in\HH)\quad
\begin{cases}
z-\sum_{j=1}^mL_j^*\overline{v_j}\in\partial f(x)+\nabla h(x)\\
(\forall i \in \{1,\ldots,m\})\quad
\overline{v_i}\in(\partial g_i\infconv \partial\ell_i)(L_ix-r_i),
\end{cases}
\end{equation}
that satisfy \ref{t:2i} and \ref{t:2ii}. It remains to show that
$\overline{x}$ solve \eqref{e:primalvar} and
$(\overline{v_1},\ldots,\overline{v_m})$ solves \eqref{e:dualvar}.
We first observe that since, for every $i\in\{1,\ldots,m\}$, 
$\dom\ell_i^*=\GG_i$ \cite[Proposition~24.27]{Livre1} yields 
\begin{equation}
\label{e:2011-06-28a}
(\forall i\in\{1,\ldots,m\})\quad
\partial g_i\infconv\partial\ell_i=\partial(g_i\infconv\ell_i).
\end{equation}
On the other hand, it follows from 
\cite[Corollary~16.38(iii) and Proposition~17.26(i)]{Livre1} that
\begin{equation}
\label{e:2011-06-28c}
\partial\big(f+h-\scal{\cdot}{z}\big)=\partial f+\nabla h-z.
\end{equation}
As a result, we derive from \eqref{e:2011-06-28b} that
\begin{equation}
\label{e:2011-06-29a}
0\in\partial\big(f+h-\scal{\cdot}{z}\big)(\overline{x})+
\sum_{i=1}^mL_i^*\big(\partial(g_i\infconv\ell_i)
(L_i\overline{x}-r_i)\big).
\end{equation}
However, since \eqref{e:2011-06-27c} and 
\cite[Proposition~16.5(ii)]{Livre1} imply that
\begin{equation}
\label{e:2011-06-27k}
\partial\big(f+h-\scal{\cdot}{z}\big)+
\sum_{i=1}^mL_i^*\big(\partial(g_i\infconv\ell_i)\big)(L_i\cdot-r_i)
\subset\partial\bigg(f+h-\scal{\cdot}{z}+
\sum_{i=1}^m(g_i\infconv\ell_i)\circ(L_i\cdot-r_i)\bigg),
\end{equation}
it follows from \eqref{e:2011-06-29a} that
\begin{equation}
\label{e:2011-06-29s}
0\in\partial\bigg(f+h-\scal{\cdot}{z}+\sum_{i=1}^m
(g_i\infconv\ell_i)\circ(L_i\cdot-r_i)\bigg)(\overline{x}).
\end{equation}
Thus, Fermat's rule \cite[Theorem~16.2]{Livre1} asserts that 
$\overline{x}$ solves \eqref{e:primalvar}. Finally, to show that
$(\overline{v_1},\ldots,\overline{v_m})$ solves \eqref{e:dualvar},
we first note that it follows from \eqref{e:2011-06-28c},
\eqref{e:heidelberg2011-07-03}, and \cite[Proposition~15.2]{Livre1}
that
\begin{equation}
\label{e:2011-06-29e}
\big(\partial f+\nabla h\big)^{-1}=\big(\partial(f+h)\big)^{-1}=
\partial(f+h)^{*}=\partial\big(f^*\infconv h^*\big).
\end{equation}
Likewise, \eqref{e:2011-06-28a} and 
\cite[Proposition~13.21(i)]{Livre1} yield
\begin{equation}
\label{e:2011-06-29f}
(\forall i\in\{1,\ldots,m\})\quad
\big(\partial g_i\infconv\partial\ell_i\big)^{-1}=
\partial\big(g_i\infconv\ell_i\big)^{*}=
\partial\big(g_i^*+\ell_i^*\big).
\end{equation}
Hence, combining \eqref{e:mdual-var1}, \eqref{e:2011-06-29e}, and 
\eqref{e:2011-06-29f}, we obtain 
\begin{equation}
\label{e:2011-06-29d}
(\exi x\in\HH)\quad
\begin{cases}
x\in\partial(f^*\infconv h^*)
\big(z-\sum_{j=1}^mL_j^*\overline{v_j}\big)\\
(\forall i\in\{1,\ldots,m\})\quad
L_ix-r_i\in\partial\big(g_i^*+\ell_i^*\big)(\overline{v_i})
\end{cases}
\end{equation}
and therefore
\begin{equation}
\label{e:2011-06-29g}
(\exi x\in\HH)\quad
\begin{cases}
-(L_ix)_{1\leq i\leq m}\in
-\bigg(\underset{i=1}{\overset{m}{\cart}}L_i\bigg)
\Big(\partial(f^*\infconv h^*)
\big(z-\sum_{j=1}^mL_j^*\overline{v_j}\big)\Big)\\
(L_ix)_{1\leq i\leq m}\in\underset{i=1}{\overset{m}{\cart}}
\partial\big(g_i^*+\ell_i^*+\scal{\cdot}{r_i}\big)(\overline{v_i}).
\end{cases}
\end{equation}
Hence, using \cite[Propositions~16.5(ii) and 16.8]{Livre1}
and the notation \eqref{e:2011-06-30a},
\begin{align}
\label{e:2011-06-29h}
(0,\ldots,0)
&\in-\bigg(\underset{i=1}{\overset{m}{\cart}}L_i\bigg)
\Bigg(\partial(f^*\infconv h^*)
\Bigg(z-\sum_{j=1}^mL_j^*\overline{v_j}\Bigg)\Bigg)+
\underset{i=1}{\overset{m}{\cart}}
\partial\big(g_i^*+\ell_i^*+\scal{\cdot}{r_i}\big)
(\overline{v_i})\nonumber\\
&=-\bigg(\bigoplus_{i=1}^mL_i^*\bigg)^*
\bigg(\partial(f^*\infconv h^*)
\bigg(z-\bigg(\bigoplus_{i=1}^mL_i^*\bigg)
(\overline{v_1},\ldots,\overline{v_m})\bigg)\bigg)
\nonumber\\
&\quad\;+\partial\bigg(\bigoplus_{i=1}^m\big(g_i^*+\ell_i^*
+\scal{\cdot}{r_i}\big)\bigg)(\overline{v_1},\ldots,\overline{v_m})
\nonumber\\
&\subset\partial\bigg((f^*\infconv h^*)
\bigg(z-\bigg(\bigoplus_{i=1}^mL_i^*\bigg)\cdot\bigg)
+\bigoplus_{i=1}^m\big(g_i^*+\ell_i^*
+\scal{\cdot}{r_i}\big)\bigg)(\overline{v_1},\ldots,\overline{v_m}).
\end{align}
In other words, by Fermat's rule, 
$(\overline{v_1},\ldots,\overline{v_m})$ solves \eqref{e:dualvar}.
Finally, the strong convergence claims in \ref{t:2iie} and 
\ref{t:2iif} follow from Theorem~\ref{t:1}\ref{t:1iie}\&\ref{t:1iif}
since the uniform convexity of a 
function $\varphi\in\Gamma_0(\HH)$ at a point of the domain of 
$\partial\varphi$ implies the uniform monotonicity of 
$\partial\varphi$ at that point \cite[Section~3.4]{Zali02}.
\end{proof}

In the following proposition we give conditions under which 
\eqref{e:2011-06-27c} is satisfied.

\begin{proposition}
\label{p:heidelberg2011-07-06}
Suppose that \eqref{e:primalvar} has at least one solution and
set 
\begin{equation}
\label{e:gennad2011-08-06a}
S=\menge{(L_ix-y_i)_{1\leq i\leq m}}
{x\in\dom f\;\text{and}\;(\forall i\in\{1,\ldots,m\})\;\:y_i\in
\dom g_i+\dom\ell_i}.
\end{equation}
Then \eqref{e:2011-06-27c} is satisfied if one of the following 
holds.
\begin{enumerate}
\item
\label{p:heidelberg2011-07-06i}
$(r_1,\ldots,r_m)\in\sri S$.
\item
\label{p:heidelberg2011-07-06ii}
For every $i\in\{1,\ldots,m\}$, $g_i$ or $\ell_i$ is real-valued.
\item
\label{p:heidelberg2011-07-06iii}
$\HH$ and $(\GG_i)_{1\leq i\leq m}$ are finite-dimensional, 
and there exists $x\in\reli\dom f$ such that 
\begin{equation}
\label{e:gennad2011-08-06b}
(\forall i\in\{1,\ldots,m\})\quad
L_ix-r_i\in\reli\dom g_i+\reli\dom\ell_i.
\end{equation}
\end{enumerate}
\end{proposition}
\begin{proof}
It follows from \eqref{e:gennad2011-08-06a} and
\cite[Proposition~12.6(ii)]{Livre1} that
\begin{align}
\label{e:gennad2011-08-06c}
S
&=\menge{(L_ix-y_i)_{1\leq i\leq m}}{x\in\dom f\;\text{and}\;
(\forall i\in\{1,\ldots,m\})\;\:y_i\in\dom(g_i\infconv\ell_i)}
\nonumber\\
&=\Menge{(L_ix-y_i)_{1\leq i\leq m}}
{x\in\dom(f+h-\scal{\cdot}{z})\;\text{and}\;
(y_i)_{1\leq i\leq m}\in
\underset{i=1}{\overset{m}{\Cart}}\dom(g_i\infconv\ell_i)}
\nonumber\\
&=\bigg(\underset{i=1}{\overset{m}{\Cart}}L_i\bigg)
\Big(\dom\big(f+h-\scal{\cdot}{z}\big)\Big)-
\dom\bigoplus_{i=1}^m(g_i\infconv\ell_i).
\end{align}

\ref{p:heidelberg2011-07-06i}:
In view of \eqref{e:gennad2011-08-06c}, 
\begin{multline}
\label{e:heidelberg2011-07-07a}
(r_1,\ldots,r_m)\in\sri S\\
\Rightarrow\quad
(0,\ldots,0)\in\sri\bigg(\bigg(\underset{i=1}{\overset{m}{\Cart}}
L_i\bigg)\Big(\dom\big(f+h-\scal{\cdot}{z}\big)\Big)-
\dom\bigoplus_{i=1}^m\,(g_i\infconv\ell_i)(\cdot-r_i)\bigg).
\end{multline}
Hence, since $(\cart_{\!i=1}^{\!m}L_i)^*=\bigoplus_{i=1}^mL_i^*$, it 
follows from \eqref{e:2011-06-28a}, \eqref{e:2011-06-28c}, and 
\cite[Theorem~16.37(i)]{Livre1} that
\begin{align}
\label{e:heidelberg2011-07-07b}
\partial f+\sum_{i=1}^mL_i^*(\partial g_i\infconv\partial\ell_i)
(L_i\cdot-r_i)+\nabla h-z
&=\partial\big(f+h-\scal{\cdot}{z}\big)+
\sum_{i=1}^mL_i^*\big(\partial(g_i\infconv\ell_i)\big)(L_i\cdot-r_i)
\nonumber\\
&=\partial\bigg(f+h-\scal{\cdot}{z}+
\sum_{i=1}^m(g_i\infconv\ell_i)\circ(L_i\cdot-r_i)\bigg).
\end{align}
Since \eqref{e:primalvar} has at least one solution it follows 
from Fermat's rule that $0$ is in the range of the right-hand 
side of \eqref{e:heidelberg2011-07-07b}, which shows that 
\eqref{e:2011-06-27c} holds.

\ref{p:heidelberg2011-07-06ii}$\Rightarrow$%
\ref{p:heidelberg2011-07-06i}: We have
$(\forall i\in\{1,\ldots,m\})$
$\dom g_i+\dom\ell_i=\GG_i$. Therefore
\eqref{e:gennad2011-08-06a} yields $S=\bigoplus_{i=1}^m\GG_i$.

\ref{p:heidelberg2011-07-06iii}$\Rightarrow$%
\ref{p:heidelberg2011-07-06i}: We have $\sri S=\reli S$. However, 
it follows from \eqref{e:gennad2011-08-06c} and
\cite[Corollary~6.15]{Livre1} that 
\begin{align}
\label{e:2011-07-08a}
\reli S
&=
\reli\bigg(\bigg(\underset{i=1}
{\overset{m}{\Cart}}L_i\bigg)\Big(\dom\big(f+h-\scal{\cdot}{z}\big)
\Big)-\dom\bigoplus_{i=1}^m\,(g_i\infconv\ell_i)\bigg) 
\nonumber\\
&=
\reli\bigg(\underset{i=1}{\overset{m}{\Cart}}
L_i\bigg)\big(\dom f\big)-\reli
\dom\bigoplus_{i=1}^m\,(g_i\infconv\ell_i)
\nonumber\\
&=
\bigg(\underset{i=1}{\overset{m}{\Cart}}
L_i\bigg)\big(\reli\dom f\big)-\underset{i=1}{\overset{m}{\Cart}}
\reli\dom(g_i\infconv\ell_i)
\nonumber\\
&=
\bigg(\underset{i=1}{\overset{m}{\Cart}}
L_i\bigg)\big(\reli\dom f\big)-\underset{i=1}{\overset{m}{\Cart}}
\big(\reli\dom g_i+\reli\dom\ell_i\big).
\end{align}
Hence $(r_1,\ldots,r_m)\in\sri S$ $\Leftrightarrow$
$(\exi x\in\reli\dom f)(\forall i\in\{1,\ldots,m\})$
$L_ix-r_i\in\reli\dom g_i+\reli\dom\ell_i$.
\end{proof}

\begin{remark}
In Problem~\ref{prob:2}, if each function
$\ell_i$ is the indicator function
of $\{0\}$, then \eqref{e:primalvar} reduces to
\begin{equation}
\minimize{x\in\HH}{f(x)+\sum_{i=1}^m\,g_i
(L_ix-r_i)+h(x)-\scal{x}{z}}.
\end{equation}
Even in this special case, the algorithm resulting from 
\eqref{e:blackpagepart3} is new. This observation remains valid 
if we further assume that $h\colon x\mapsto 0$.
\end{remark}

\end{document}